\documentclass{amsart}
\usepackage{amsmath} 
\usepackage{graphicx} 
\usepackage{amsfonts}
 \usepackage{amssymb}

\newtheorem{theorem}{Theorem} 
\theoremstyle{plain}

 \newtheorem{corollary}{Corollary}

 \newtheorem{proposition}{Proposition} 
 \newtheorem{remark}{Remark}

 \numberwithin{equation}{section}

\begin{document}

\title[The loop cohomology of a space with the polynomial] {The loop cohomology of a space with the polynomial cohomology algebra } \author{Samson
Saneblidze} \address{A. Razmadze Mathematical Institute\\ Department of Geometry and Topology\\ M. Aleksidze st., 1\\ 0193 Tbilisi, Georgia}
 \email{sane@rmi.ge}

\thanks{This research described in this publication was made possible in part by the grant \\ GNF/ST06/3-007 of the Georgian National Science
Foundation}



\date{}

\begin{abstract} Given a simply connected space $X$ with the cohomology  $H^*\!(X;{\mathbb Z}_2)$ to be polynomial,
 we  calculate the loop cohomology algebra $H^*(\Omega X;{\mathbb Z}_2)$  by means  of the action of the Steenrod cohomology operation $Sq_1$ on
 $H^*(X;{\mathbb Z}_2).$
As a consequence we obtain  that $H^*(\Omega X;{\mathbb Z}_2)$ is the exterior algebra
 if and only if
$Sq_1$ is multiplicatively decomposable on $H^{\ast}(X;{\mathbb Z}_2).$ The last statement in fact contains a converse of a theorem of A. Borel.

\end{abstract}

\maketitle

\section{Introduction}

Let  $X$ denote a simply connected topological space. The cohomology $H^*(X)$ is considered in  coefficients $\mathbb{Z}_2={\mathbb Z}/2{\mathbb Z}$
unless otherwise specified explicitly.
 In \cite{borel}, A. Borel gave a condition for  $H^*(X)$
  to be polynomial in terms of \emph{a simple system of generators} of the loop
space cohomology $H^*(\Omega X)$ that are transgressive (see also \cite{Mo-Ta},\,\cite{McCleary2}). This was one of the first nice applications of
spectral sequences that has been introduced in \cite{Serre}, and led in particular  to calculations of the cohomology of the Eilenberg-MacLane
spaces (see \cite{McCleary2}). However, for the converse direction, that is to determine $H^*(\Omega X)$ as an  algebra for a given $X$ with
$H^*(X)$ polynomial, a spectral sequence argument no longer works. On the other hand, it was known \cite{HMS} that there is an additive isomorphism
$H^*(\Omega X)\approx H^*(BH(X))$ where $BH(X)$ denotes the bar construction of $H(X).$ In the case   the shuffle product on $BH(X)$ is
\emph{geometric} we would get $H^*(\Omega X)$ to be exterior, but this is not true in general (cf. \cite{proute}).

In this paper we completely calculate the algebra $H^*(\Omega X)$ for $H^*(X)$ polynomial by means  of the Steenrod cohomology operation $Sq_1$ on
$H^*(X)$ (Theorem \ref{omegapoly}) and then establish the criterion  for $H^*(\Omega X)$ to be exterior (Corollary \ref{1}). Namely, given
$H(X)=H(C^*(X),d)$       recall that \[Sq_1:H^n(X)\rightarrow H^{2n-1}(X)\]
 is defined for $y\in H^n(X),\, y=[c],\,c\in C^n(X),dc=0,$ by
 $Sq_1(y)=[c\smile_1 c].$
Let now  $H^*(X)=\mathbb{Z}_{2}[y_1,\!...,y_k,\!...].  $ Suppose that  a set    $\mathcal{H}=\{y_k\}$ of polynomial generators of $H(X)$ is chosen
such that
 either   $Sq_1(y_i)\in H^{+}\cdot H^{+}$  or   $Sq_1(y_i)=y_k\!\! \mod H^{+}\cdot H^{+} $ where $y_i$ is uniquely determined
 by a given $y_k.$
  Define a subset $\mathcal{S}\subseteq \mathcal{H}$   as
  \[\mathcal{S}=\{ z_s\in \mathcal{H}\mid z_s\notin \operatorname{Im}Sq_1\!\!\mod H^{+}\cdot H^{+}\}.\]
Thus $\mathcal{S}= \mathcal{H}$ if and only if  $Sq_1(y_k)\in H^{+}\cdot H^{+}$ for all $k.$ Let $0\leq \nu_i<\infty$ be the smallest integer such
that $Sq_1^{(\nu_i+1)}(y_i)\in H^{+}\cdot H^{+},$ where $Sq_1^{(m)}$ denotes the $m$-fold composition $Sq_1\circ \cdots \circ Sq_1.$
 The integer $\nu_i$ is referred to as
the \emph{weak $\smile_1$-height} of $y_i;$ when the finite integer $\nu_i$ does not exist
  we say that $y_i$  has the infinite  weak  $\smile_1$-height  $\nu_i=\infty.$
  Let $\sigma : H^*(X)\rightarrow H^{*-1}(\Omega X)$ be the suspension homomorphism.
\begin{theorem}\label{omegapoly} Let
 $X$ be a simply connected space with $H^*(X)=\mathbb{Z}_{2}[y_1,\!...,y_k,\!...] $
 and $\nu_k$ to be the weak $\smile_1$-height of $y_k.$
Then the algebra $H^*(\Omega X)$ is multiplicatively generated    by the elements
 $\bar z_s=\sigma z_s$   satisfying the only  relations ${\bar z}_s^{2^{\nu_s+1}}=0$
 for $z_s\in \mathcal{S} .$
\end{theorem} \begin{corollary}\label{1} $H^*(\Omega X)=
 \Lambda(\bar y_1,\!...,\bar y_k,\!...)$ is the exterior algebra if and only if
 $y_k$ is of the zero  weak  $\smile_1$-height, i.e.,
  $Sq_1(y_k)\in H^{+}\cdot H^{+},$ for all $k.$
\end{corollary} \begin{corollary} $H^*(\Omega X)=
 {\mathbb Z}_{2}[\bar z_1,\!...,\bar z_s,\!...]$ is the polynomial algebra if and only if  $z_s$ is of the infinite  weak  $\smile_1$-height  for
 all $s.$
\end{corollary}

Our method of proving the theorem consists of using  the \emph{filtered Hirsch} model $(RH,d+h)\rightarrow C^*(X)$ of $X$ \cite{saneFiltered}. Note
that the underlying differential (bi)graded algebra $(RH,d)$ is a non-commutative version of  Tate-Jozefiak resolution of the commutative algebra
$H$ (\cite{Tate},\,\cite{Jozefiak}), while $h$ is a perturbation of $d$ \cite{sane} similar to \cite{hal-sta}. Furthermore,  the tensor algebra
$RH=T(V)$ is endowed with higher order operations $E=\{E_{p,q}\}$ that extend $\smile_1$-product measuring the non-commutativity of the product on
$RH;$ and there also is a binary operation $\cup_2$ on $RH$ measuring the non-commutativity of the $\smile_1$-product. In general, by means of
$(RH,d+h)$ one can recognize the cohomology $H(BC^*(X))$ of the bar construction $BC^*(X)$ as an algebra.
 The case of polynomial
$H$  is distinguished  since $H$ has no multiplicative relations unless that of the commutativity;  furthermore,  we  can equivalently  take a
small
 multiplicative resolution $R_{\tau}H$ in which the Hirsch algebra structure is given by commutative (on $V_{\tau}$) and associative
 $\smile_1$-product.
This  allows an explicit calculation of the algebra $H(BC^*(X)),$ and, consequently, of the loop space cohomology $H^*(\Omega X)$ in question.

Obviously the hypothesis of Corollary \ref{1} is satisfied for an evenly graded polynomial algebra $H^*(X).$ Note that our method  can be in fact
applied to an evenly  graded polynomial algebra  $H^*(X;\Bbbk)$ for any coefficient ring $\Bbbk$ to establish that $H^*(\Omega X;\Bbbk)$ is
exterior. Though, this fact
   can be also  deduced
 from the Eilenberg-Moore spectral sequence (see,
 for example, \cite{McCleary2};
 for further  references  of spaces with polynomial cohomology rings see  also \cite{notbohm}, \cite{Ande-Grod}).

\vspace{0.2in} \section{Hirsch resolutions of  polynomial algebras} We adopt the notations  and terminology of \cite{saneFiltered}. Recall that
given  a Hirsch algebra $(A,\{E_{p,q}\})$ with $H=H^*(A),$ there is
  a filtered Hirsch model
\begin{equation*}
 f: (RH,d_h)\rightarrow A.
\end{equation*}
 In general
 $\rho: (RH,d)\rightarrow H$ is a  multiplicative resolution of the graded commutative algebra $H$
  with
 $R^*H^*=T(V^{*,*})$ (bigraded tensor algebra) and
 \begin{equation*}
V^{*,*}= \mathcal{E}^{*,*} \oplus \mathcal{T}^{*,*}\oplus {\mathcal M}^{*,*}. \end{equation*} Whence  multiplicative generators of $H$ is chosen
the module
  ${\mathcal M}^{0,*}=V^{0,*}$
 is uniquely determined;  furthermore,  ${\mathcal M}^{<0,*}$
  corresponds to (multiplicative) relations in  $H$ which is not a consequence of that of the commutativity, while
 ${\mathcal E}$ just corresponds to the commutativity relation in $H.$
 (The module ${\mathcal T}$ is determined  by $\cup_2$-product that measures the non-commutativity of $\smile_1$-product.)

  However, when $H={\mathbb Z}_2[y_1,...,y_k,...]$ is polynomial,
  the module $\mathcal M$ is much simplified since $H$ has no relations unless that of  the commutativity. Namely, we can  set  ${\mathcal
  M}^{<0,*}=0.$ In particular,
 denoting a basis element of $V^{0,*}$ by  $x_k,$ i.e.,
       ${\mathcal V}^{0,*}=\{ x_k\}$ with $\rho x_k=y_k,$
    we have that a basis of the kernel of the epimorphism $\rho|_{R^0H}:R^0H\rightarrow H$ is formed  by  $x_ix_j+x_jx_i,\,i\neq j,$
   and
  \begin{multline*}
  V^{-1,*}=\mathcal{E}^{-1,*}=\langle x_i\smile_1x_j \mid x_k\in {\mathcal V}^{0,*}  \rangle
  \ \
   \text{with} \\
    d(x_i\smile_1x_j)=d(x_j\smile_1x_i)=x_ix_j+x_jx_i.
  \end{multline*}
  \begin{multline*}(\mathcal{T}^{-2,*}=\langle x_i\cup_2x_j\,(= x_j\cup_2x_i) \mid x_k\in {\mathcal V}^{0,*}  \rangle \ \ \text{with}\\
      d(x_i\cup_2x_j)=x_i\smile_1x_j+x_j\smile_1x_i,i\neq j,\ \ \text{and}\ \  d(x_i\cup_2x_i)=x_i\smile_1 x_i. )
      \end{multline*}
 Moreover, we can go further
and reduce $V$  at the cost of $\mathcal{E}$ (and, consequently, of ${\mathcal T}$ too) to obtain a small multiplicative resolution $R_{\tau}H.$
 Namely,
 set
  \[R_{\tau}H=RH/ J_{\tau} \]
 where $J_{\tau}\subset RH$ is a Hirsch ideal generated by
   \begin{multline*}\{ E_{p,q}(a_1,...,a_p;a_{p+1},...,a_{p+q}),\,dE_{1,2}(a_1; a_2, a_3)
   ,\,dE_{2,1}(a_1,a_2; a_3)
   ,\, a\cup_2 b,\,d(a\cup_2b) \\
   |\,  (p,q)\neq(1,1),\,
   a,b\in {\mathcal V},\, a\neq b  \}
   \end{multline*}
   where   $a_i\in RH$ unless $i=p+q$ for  $p\geq 2$ and $q=1$ in which case $a_{p+1}\in {\mathcal V}.$
   Since $d: J_{\tau}\rightarrow  J_{\tau},$ we get a Hirsch algebra map $g_{\tau}:(RH,d){\rightarrow} (R_{\tau}H,d)$  so  that
    a resolution map $\rho: RH \rightarrow H$   factors  as
   \[\rho: (RH,d)\overset{g_{\tau}}{\longrightarrow} (R_{\tau}H,d) \overset{\rho_{\tau}}{\longrightarrow} H.\]
    By definition we have $h:{\mathcal E}\rightarrow {\mathcal E};$ furthermore, since the transgressive component $h^{tr}$ of $h$ annihilates
    $a\cup_2b$ for $a\neq b,\,a,b\in {\mathcal V}$
     (cf. \cite[Proposition 4]{saneFiltered}),
     we  get
 $h: J_{\tau}\rightarrow  J_{\tau},$ too. Thus
   $g_{\tau}$  extends to  a quasi-isomorphism  of Hirsch algebras
 \begin{equation*}
 g_{\tau}:(RH,d_h)\rightarrow (R_{\tau}H,d_h).
 \end{equation*}

Note that  the Hirsch algebra $(R_{\tau}H,d_h)$ can be described immediately  as follows. Indeed, we have $R_{\tau}H=T(V^{*,*}_{\tau})$ with
$V^{*,*}_{\tau}=\langle {\mathcal V}^{*,*}_{\tau} \rangle,$
 \begin{multline*}
 {\mathcal V}_{\tau}=\{x_i,\, {x_j}^{\cup_2 2^{m}},\,      a_{i_1}\!\smallsmile_1\cdots \smallsmile_1 \! a_{i_n}
 \mid a_{i_r}\in\{x_i, {x_{j}}^{\cup_2 2^{m}}\}_{m\geq 1},\,  x_k\in {\mathcal V}^{0,*},n\geq 2\}.
  \end{multline*}
The differential $d$ on $R_{\tau}H$ is determined by \begin{multline*}dx_k=0,\,
 d(a\!\smile_1b)=da\!\smile_1\!b +a\!\smile_1\! db+ab+ba \ \  \text{and}\\
 d(x_k\cup_2 x_k)=x_k\smile_1x_k,\,\,\,  d(x_k^{\cup_2{2^{m}}})=x_k^{\cup_2{2^{m-1}}}\!\!\smile_1 \!x_k^{\cup_2{2^{m-1}}},\,m\geq 2,
  \end{multline*}
   while its perturbation $h$ by
  \[
  hx_k=0,\,h(a\smile_1b)=ha\smile_1b +a\smile_1 hb\ \
     \text{and} \ \
     h(x_k^{\cup_2{2^{m}}})= h^{tr}(x_k^{\cup_2{2^{m}}}),
           \]
      where
      $ h^{tr}(x_k\cup_2 x_k)$ is defined by
       $\rho_{\tau} h^{tr}(x_k\cup_2 x_k)=Sq_1(y_k).$
         The Hirsch  algebra structure   of  $(R_{\tau}H,d_h)$ is generated by  commutative (on $V_{\tau}$) and  associative $\smile_1$-product
         satisfying the  (left) Hirsch formula
\begin{equation*} c\smile _{1}ab=(c\smile _{1}a)b+a(c\smile _{1}b)
  \end{equation*}
  and the (right) generalized Hirsch formula for $c\in V_{\tau}$
  \begin{equation*}
ab\smile_1 c=\left\{\!\!\begin{array}{llll} a(b\smile_1 c)+ (a\smile_1 c)b, &  c\in \{x_i, {x_{j}}^{\cup_2 2^{m}}\}_{m\geq 1},\vspace{5mm}\\
 a(b\smile_1 c)+(a\smile_1 c)b\\+\,(a\smile_1 c_1)(b\smile_1 c_2)+(a\smile_1 c_2)(b\smile_1 c_1) , &   c=c_1\smile_1 c_2.
\end{array} \right. \end{equation*}

To ensure that $\rho_{\tau}:(R_{\tau}H,d)\rightarrow H$ is a multiplicative resolution of $H$ it suffices to verify the following

\begin{proposition}\label{acyclic} The  chain complex $(R_{\tau}^*H^*,d)$ is acyclic in the negative resolution degrees, i.e.,
$H^{i,*}(R_{\tau}^iH^*,d)=0,i<0.$ \end{proposition} \begin{proof} First observe that as a cochain complex
 $\operatorname{Ker}\rho_{\tau}$ can be decomposed  via $\operatorname{Ker} \rho_{\tau}=A\oplus B$ in which
 $A=\oplus A(n),\,n\geq 2,$ $A(n)$ is determined by all monomials  consisting of the $\cdot $ and
$\smile_1$ products evaluated on all generators $x_{i_1},...,x_{i_n}\in V^{0,*}_{\tau}$  with distinct $x_i$'s and $B$ is determined by the other
monomials. For example, the cochain complex $A$ is acyclic since $A(n)$ can be identified with the cellular chains of the permutohedron $P_n$ (cf.
\cite{saneClass}); thus a chain contracting homotopy $s_A:A\rightarrow A$ is fixed. To see that $B$ is also acyclic, define a map $s_B:B\rightarrow
B$ of degree $-1$ as follows. On $B^{0,*}:$ For $xy\in B^{0,*}$ with  $x\in A^{0,*}$ or $y\in A^{0,*},$  set $s_B(xy)=s_A(x)y$ or $s_B(xy)=xs_A(y).$
On $B^{<0,*}:$
 For a monomial $u=u_1\cdots u_m\in B^{<0,*},$
let $i$ be the first integer such that either
  $u_i=
x_{i_1}\!\smallsmile_1x_{i_2}\!\smallsmile_1\cdots \smallsmile_1 \! x_{i_n},\,n\geq 2,$ or $u_i= x^{\cup_2 2^{k}}_{i_1}\!\smallsmile_1 x^{\cup_2
2^{k}}_{i_2}\!\smallsmile_1 y,$  $k\geq 1,\,y\in V$ where $i_1=i_2.$ Set $s(u)=u_1\cdots \tilde u_i\cdots u_n$ with $\tilde u_i=x_{i_1}\!\cup_2
x_{i_1}\!\smallsmile_1x_{i_3}\!\smallsmile_1\cdots \smallsmile_1 \! x_{i_n}$ or $\tilde u_i=x^{\cup_2 2^{k+1}}_{i_1}\!\!\smallsmile_1 y$
respectively.
  While define $s_B$ to be zero on the other monomials of $B^{<0,*}.$
Then  for each element $a\in B $ there is an integer $n(a)\geq 1$ such that $n(a)^{th}$-iteration of the operator $s_Bd+ds_B-Id:B\rightarrow B$
evaluated on $a$ is zero, i.e.,
  $  (s_Bd+ds_B-Id)^{(n(a))}(a)=0$ as desired.
\end{proof}

\section{Proof of Theorem 1}

Recall   (\cite{baues3},\,\cite{KScubi}) that  given a space  $X,$ there are operations $E=\{E_{p,q}\}$  on the cochain complex $C^*(X)$ making it
into a Hirsch algebra. Note that in the simplicial case one can choose $E_{p,q}=0$ for $q\geq 2.$ Furthermore,  given a Hirsch algebra $A,$ its
structural operations $E=\{E_{p,q}\}$  induce a product $\mu_{E}$ on the bar construction $BA.$ In particular,
  there is an algebra
isomorphism \begin{equation*} H^*(\Omega X) \approx H(BC^{\ast}(X),d_{_{BC}},\mu_{_{E}}). \end{equation*}
 (In the above we assume $C^*(X)=C^*(\operatorname{Sing}^1X)/C^{>0}(\operatorname{Sing}\, x),$ in which
 ${\operatorname{Sing}}^{1}X\subset
{\operatorname{Sing}}X$ is the Eilenberg 1-subcomplex generated by the singular simplices that send the 1-skeleton of the standard $n$-simplex
$\Delta ^{n}$ to the base point $x$ of $X.$) \begin{proposition}\label{comparison} A morphism $g: A\rightarrow A'$ of Hirsch algebras induces a
Hopf dga map of the bar constructions \[Bg : B A\rightarrow B A'\] and  if  $g$ is a homology isomorphism, so is $Bg.$ \end{proposition}
\begin{proof} The proof is standard by using a spectral sequence comparison argument. \end{proof} Denote $\bar{V_{\tau}}=s^{-1}(V_{\tau}^{>0})\oplus
{\mathbb Z}_2 $ and define the differential
 $\bar {d}_h$ on $\bar{V}_{\tau}$     by the restriction of $d+h$ to $V_{\tau}$
to obtain the cochain complex $(\bar{V}_{\tau},\bar{d}_h).$ Let  $\psi: B(R_{\tau}H)\rightarrow \bar V_{\tau}$ be the standard projection of cochain
complexes. We can convert $\psi$ as a map of dga's by introducing  a product on   $\bar V_{\tau}.$ Namely,  for $\bar a, \bar b\in \bar V_{\tau},$
set \[\bar a\bar b=\overline {a\smile_1b} \ \ \  \text{with}\ \ \     \bar a1=1\bar a=\bar a. \]
 Then we get  the following sequence of algebra isomorphisms
\begin{multline*} H(BC^{\ast}(X),d_{_{BC}},\mu_{_{E}})\overset{Bf^{\ast}}{\underset{\approx}{\longleftarrow}} H(B(RH),d_{_{B(RH)}},\mu_{_{E}})
 \overset{Bg_{\tau}^{\ast}}{\underset{\approx}{\longrightarrow}}
H(B(R_{\tau}H),d_{_{B({R}_{\tau}H)}},{\mu}_{_{{E}_{\tau}}})\\ \overset{\psi^*}{\underset{\approx}{\longrightarrow}} H(\bar V_{\tau}, \bar d_h),
\end{multline*} where the first two isomorphisms are by Proposition \ref{comparison}, while the third isomorphism (additively) is a consequence of a
general fact about tensor algebras
 \cite{F-H-T} (see also \cite{HMS}).
 Thus  the calculation of the algebra
$H(\Omega X)$ reduces to that of $H(\bar V_{\tau},\bar d_{h}).$ By definition of $h$ it is easy to see that any $\bar d_h$-cocycle in $\bar
V^{*,*}_{\tau}$ is cohomologous to a $\bar d_h$-cocycle in $\bar V_{\tau}^{0,*}.$ In particular
 $\bar x_k^{2^m}(=s^{-1}\left(x_k^{\smallsmile_1 2^{m}}\right))$ is cohomologous to $s^{-1}\left(a_k|_{V_{\tau}^{0,*}}\right)$
 for $a_k\in R^0H$
with $\rho_{\tau}a_k=Sq_1^{(m)}(y_k),$  so  the cohomology algebra  $H(\bar V_{\tau},\bar d_{h})$ is as desired.

\begin{remark} Refer  to Example 3 from \cite{saneFiltered} and recall that there is a canonical Hirsch algebra structure $Sq=\{Sq_{p,q}\}$ on
$H(X)$ determined by $Sq_1.$     The isomorphism $H^*(\Omega X)\approx H^*(BH(X))$ from the introduction converts into  an algebra one  when
$BH(X)$  is endowed with the product $\mu_{_{Sq}}.$ Details are left to the interested reader. \end{remark} 


\vspace{0.1in}

\begin{thebibliography}{99}

\bibitem{Ande-Grod} K.K.S. Andersen and J. Grodal,  The Steenrod problem of realizing polynomial cohomology rings,  J. Topology, 1 (2008),
    747--460.


\bibitem{baues3}  H. J. Baues, The cobar construction as a Hopf algebra, Invent. Math., 132 (1998), 467--489.


\bibitem{borel} A. Borel, Sur la cohomologie de espaces fibr\'{e}s et  des
 espaces homogenes de groupes de Lie compacts, Ann. of  Math.,
                57 (1953), 115--207.

\bibitem{F-H-T} Y. Felix, S. Halperin and  J.-C. Thomas,
 Adams' cobar equivalence, Trans. AMS, 329 (1992), 531--549.



\bibitem{hal-sta} S. Halperin and J.  Stasheff, Obstructions to homotopy equivalences, Adv. in Math., 32 (1979),  233--279.

\bibitem{HMS} D. Husemoller, J.C. Moore and J. Stasheff, Differential homological algebra and homogeneous spaces, J. Pure and Applied Algebra, 5
    (1974), 113--185.


\bibitem{Jozefiak} J. T. Jozefiak, Tate resolutions
 for commutative graded algebras over a local ring, Fund. Math., 74 (1972),
 209--231.


\bibitem{KScubi} T. Kadeishvili and S. Saneblidze, A cubical model of a fibration, J. Pure and Applied Algebra, 196 (2005), 203--228.


\bibitem{McCleary2} J. McCleary,  \textquotedblleft User's guide to spectral sequences \textquotedblright (Publish or Perish. Inc., Wilmington,
    1985).


\bibitem{Mo-Ta} R. E. Mosher and M. C. Tangora,  \textquotedblleft Cohomology operations and applications in homotopy theory \textquotedblright
    (Dover Publ. 2008).

\bibitem{notbohm} D. Notbohm,
 \textquotedblleft Classifying spaces of compact Lie groups and finite loop
spaces \textquotedblright, Handbook of algebraic topology (Ed. I.M. James), Chapter 21 (North-Holland,
     1995).

\bibitem{proute} A. Prout\'{e}, Un contre-example $\grave{a}$ la g\'{e}om\'{e}tricit\'{e} du shuffle-coproduit de la cobar-construction, C.R. Acad.
    Sc. Paris, 298, s\'{e}rie I (1984), 31--34.

\bibitem{sane}  S. Saneblidze, Perturbation and obstruction theories
 in fibre spaces,  Proc.  A. Razmadze Math. Inst., 111 (1994),  1--106.

\bibitem{saneFiltered} S. Saneblidze, Filtered Hirsch algebras, preprint, math. AT/0707.2165.


\bibitem{saneClass}  S. Saneblidze,  On the homotopy classification of maps, J. Homotopy and Rel. Struc., 4 (2009), 347--357.


\bibitem{Serre} J.-P. Serre, Homologie singuli\`{e}re des espaces fibr\'{e}s. Applications, Ann. Math., 54 (1951), 425--505.


\bibitem{Tate} J. Tate, Homology of noetherian rings and local rings, Illinois J. Math., 1 (1957), 14--27.


\end{thebibliography}
\end{document}